\documentclass[12pt, a4paper]{amsart}
\usepackage{mathrsfs, amssymb, eucal, amsfonts, amsmath, amsthm, amscd}
\numberwithin{equation}{section}

\newcommand{\Bc}{{\mathcal B}}
\newcommand{\Dc}{{\mathcal D}}\newcommand{\Ec}{{\mathcal E}}

\newcommand{\As}{{\mathscr A}}
\newcommand{\Bs}{{\mathscr B}}\newcommand{\Cs}{{\mathscr C}}

\newcommand{\Ms}{{\mathscr M}}

\newcommand{\Hs}{{\mathscr H}}\newcommand{\Ks}{{\mathscr K}}

\newcommand{\Mg}{{\mathfrak M}}

\newcommand{\R}{{\mathbb R}}

\newcommand{\bgam}{\boldsymbol{\gamma}}
\newcommand{\bM}{\boldsymbol{M}}

\begin{document}
\newtheorem{defn}{\sc Definition}[section]
\newtheorem{teo}{\sc Theorem}[section]
\newtheorem{lem}[teo]{\sc Lemma}
\newtheorem{prop}[teo]{\sc Proposition}
\newtheorem{cor}[teo]{\sc Corollary}
\theoremstyle{remark}
\newtheorem{ex}{\bf Example}[section]
\newtheorem{rem}{\bf Remark}[section]

\title[Stochastic Mappings and Random Distribution Fields]{Stochastic Mappings and Random Distribution Fields. A Correlation Approach}

\author{P\u{a}storel Ga\c{s}par}
\address{Department of Mathematics and Computer Science,
Faculty of Exact Sciences,
``Aurel Vlaicu'' University, Arad,
Str. E. Dr\u agoi nr. 2,
310330  Arad,
Romania }
\email{pastorel.gaspar@uav.ro}

\author{Lorena Popa}
\email{popa.lorena@yahoo.com}

\date{\today}

\begin{abstract}
This paper contains a study of multivariate second order stochastic mappings indexed by an abstract set $\Lambda$ in close connection to their operator covariance functions. \\
The characterizations of the normal Hilbert module or of Hilbert spaces associated to such a multivariate second order stochastic mapping in terms of reproducing kernel structures are given, aiming not only to gather into a unified way some concepts from the field, but also to indicate an instrument for extending the very well elaborated theory of multivariate second order stochastic processes (or random fields) to the case of multivariate second order random distribution fields, including multivariate second order stochastic measures.  In particular a general Wold type decomposition is extended and discussed in our framework.
\end{abstract}
\maketitle

\section{Introduction}

The study of stochastic or random processes (i.e. a family of random variables) is nowadays among the important topics in mathematical research.
The family of random variables composing a stochastic process was, naturally, first indexed by the set of integers ${\mathbb Z}$ and then by the set of real numbers ${\mathbb R}$. Further the necessity of simultaneous study of more stochastic processes or to describe more complicated phenomena from technical, natural or social sciences,  led to consider the process as consisting of multivariate (finite or infinite dimensional) random variables or even generalized random variables (i.e. continuous linear operators sending the complex random variables into some Hilbert space) and, on the other side, to index the process by ${\mathbb Z}^d$ and ${\mathbb R}^d$ (with $d$ a positive integer) or, even by an abstract semigroup or other algebraic-topological structures.

In an attempt to gather the existing frameworks for studying random fields we introduce \emph{multivariate stochastic mappings} as families of $H$-valued random variables indexed by an arbitrary set $\Lambda$ (often taken with a topological structure), with $H$ an infinite dimensional separable complex Hilbert space.


The plan of the paper runs as follows. In Section 2 the basic concepts regarding the organization of the set of second order $H$-valued random variables into a normal Hilbert $\Bc(H)$-module and some other basic notations are introduced, the outline of the chosen framework and it's connections to the developments from the specialized literature being mentioned. \\
Section 3 presents basic results regarding the correlation theory (including a Kolmogorov-type factorization), formulated in the general setting of multivariate second order stochastic mappings, where reproducing kernel techniques are essentially used. \\
In Section 4 multivariate second order random distribution fields are defined and using the space $\Dc_d$ of test functions on $\R^d$ from distribution theory as index set, the results from Section 3 are applied to obtain a correlation theory for such type of stochastic mappings. A general Wold decomposition in this extended framework ends this Section. \\
Since, in further developments of the theory of multivariate random distribution fields their subclass of regular multivariate stochastic not necessarily bounded measures on $\R^d$ play an important role, Section 5 is devoted to transpose  the results of Sections 3 and 4 to such measures as stochastic mappings indexed by some $\delta$-rings of Borel sets in $\R^d$. \\
Finally in Section 6 considerations about next steps in developing the theory in this framework are made.



\section{Preliminaries and basic notations}

We start with a probability space $(\Omega, \As, \wp)$ and denote by $L^{2}(\wp)$ the Hilbert space of second order complex valued random variables and by $L^{2}_s(\wp, H)$ the Hilbert space of strong second order $H$~-~valued random variables. More generally, a function $f:\Omega\rightarrow H$ is from $L^{r}_s(\wp,H),\ 1\le r < \infty$, if it is strongly measurable (in this case equivalent to weakly measurable) and the $r$th power of its norm is integrable, the spaces being endowed with the usual $r$-norms. Hence, the scalar product in $L^{2}_s(\wp,H)$ takes the natural form:

\begin{equation}\label{111}
(f,g)_{L_s^{2}(\wp,H)}:=\int\limits_{\Omega}\left(f(\omega), g(\omega)\right)_{H}d\wp(\omega) \ ;\quad f,g \in L^{2}_s(\wp,H)
\end{equation}
(see also \cite{Kaki} pp.11).
The space $L^2_w(\wp, H)$ of weak second order $H$-valued random variables being in an obvious way defined, it is not hard to see that $L^2_s(\wp, H) \subset L^2_w (\wp, H) \subset L^1_s (\wp, H)$. Then we can identify in $L^{2}_s(\wp,H)$ and also in $L^2_w(\wp, H)$ the subspaces of strong, respectively weak second order $H$-valued random variables of zero mean \footnote{Since the random variables composing a stochastic process or random field are supposed to have the same mean, we might as well consider only the ones of zero mean.} $L^2_{s, 0} (\wp, H)$ and $L^2_{w,0}(\wp, H)$ respectively.
Now, the mapping
\begin{equation}\label{eq:coresp.weak.rnd.var}
f\mapsto V_f
\end{equation}
where
\begin{equation} \label{eq:V_f}
V_f \chi = \int_\Omega \chi (\omega) f(\omega)  d \wp (\omega) ,\ \ \ \chi \in L^2_0(\wp),
\end{equation}
enables us to regard the elements of these two spaces as Hilbert space morphisms between $L^2_0(\wp)$ and $H$. This is done by endowing them with a structure of inner product $\Bc(H)$-modules. Therefore we establish the following notations. When $G$ is another complex Hilbert space, then $\Bc(G, H)$ means the space of continuous linear operators, while $\Cs_2(G,H)$ is the class of Hilbert-Schmidt operators from $G$ to $H$. $\Cs_1(H)$ is the ideal of trace class operators from the $C^*$-algebra $\Bc(H) = \Bc(H, H)$.
Now, $L^2_{s,0}[\wp, H]$ denotes $L^2_{s, 0} (\wp, H)$ organized as a normal Hilbert $\Bc(H)$-module with the natural outer action of $\Bc(H)$ and the $\Cs_1(H)$-valued inner product (Gramian) defined \footnote{here the tensor product is understood in the sense of Schatten, i.e. $(h\otimes \bar k)h':\ = (h',k)_H h;\ h,h',k \in H$} by
\begin{equation}\label{eq:gramian.int}
[f,g]_{L^2_{s,0}[\wp, H]}:\ = \int f(\omega) \otimes \overline{g(\omega)} d\wp(\omega), \ f, g \in L^2_{s,0}[\wp, H],
\end{equation}
and the topology given by the norm
\begin{equation}\label{eq:norm.mod.Hilb}
\| f\|_{L^2_{s,0}[\wp, H]} = \| [f,f]\|^{1/2}, \qquad f \in L^2_{s,0}[\wp, H].
\end{equation}
Let's note that because of the connection $(f,g)_{L^2_{s,0}(\wp, H)} = {\mathrm {tr}}[f,g]_{L^2_{s,0}[\wp, H]}$  between the scalar product \eqref{111} and the Gramian \eqref{eq:gramian.int},
the norm \eqref{eq:norm.mod.Hilb}, defined above, actually coincides with the natural $L^2$-norm.

Thus, the mapping \eqref{eq:coresp.weak.rnd.var} is a (Gramian preserving) module isomorphism of $L^2_{s,0}[\wp , H]$ onto the normal Hilbert $\Bc(H)$-module $\Cs_2(L^2_0(\wp), H)$, while its extension to $L^2_{w,0}(\wp, H)$ embeds it in a natural way into the Hilbert $\Bc(H)$-module $\Bc(L^2_0(\wp), H)$.

The elements of $\Hs$, which as in \cite{Kaki} represents henceforth a common notation for $L^2_{s,0}[\wp,H]$ and $\Cs_2(L^2_0(\wp), H)$, respectively the elements of
$\Bc(L^2_{0}(\wp),H)$ (including those from $L^2_{w,0}(\wp, H)$) will be called \emph{multivariate strong second order random variables of zero mean}, respectively \emph{generalized multivariate second order random variables of zero mean}.
In what follows instead of \emph{strong second order} we use simply the term \emph{second order}.

Generally speaking a (generalized) multivariate stochastic process is a family $\{\Phi_{\lambda}\}_{\lambda\in\Lambda}$ of (generalized) multivariate random variables, where if $\Lambda = \mathbb{Z}^{d}$, we have a \emph{d-time discrete parameter (generalized) multivariate stochastic process} or, if $\Lambda = \mathbb{R}^{d}$, we have a \emph{d-time continuous parameter (generalized) multivariate stochastic process}. \\
Since, in our study we shall use Hilbert space operator methods, $\Phi_{\lambda},\linebreak \lambda\in\Lambda$ will be required to be strongly (or weakly) square integrable and, without restraining the generality, of zero mean, i.e. $\Phi_\lambda \in L^2_{s,0} (\wp, H)$ (or $\Phi_\lambda \in L^2_{w,0} (\wp, H)\subset \Bc(L^2_0(\wp), H)$). So, in these cases we speak about
$\{ \Phi_\lambda\}_{\lambda \in \Lambda}$ as being a \emph{(generalized) multivariate second order stochastic process}.
For an embedding as before, see also \cite{Weron} or \cite{ChoWeDiss}, where Banach space valued random variables are considered and the term of \emph{generalized second order stochastic process} in this sense first appears.

When $d>1$, or even when $\Lambda$ is a locally compact abelian group, the term \emph{(generalized) random field} is preferred instead of \emph{(generalized) stochastic process}, while
when $\Lambda$ is an arbitrary index set, the term \emph{(generalized) stochastic mapping} is the most appropriate.

So, the basic concept which we shall use in this paper will be a function $\Phi$ defined on an arbitrary set $\Lambda$,
\begin{equation}\label{112}
\Lambda\ni \lambda\mapsto \Phi(\lambda) \in {\mathscr H} \left(\text{ or }\Bc(L^2_0(\wp), H)\right),
\end{equation}
which we shall call \textit{(generalized) multivariate second order stochastic mapping}, briefly (g.)m.s.o.s.m.

We shall mention further how several such concepts from the scientific literature fit in our present setting. When $\Lambda$ is a separable metric space (which can automatically be regarded as a measurable space with respect to a positive measure on the $\sigma$-algebra $\Sigma = {\mathcal B}or\Lambda$ of Borel sets), then we find ourself with multivariate stochastic mappings in the framework of \emph{infinite dimensional random mappings} from the paper \cite{Thang2}.
Let us mention that in this framework regularity conditions (some kinds of continuity or measurability) for m.s.o.s.m. occur naturally. They can also be formulated even in some enlarged frameworks. For instance, if the parameter set $\Lambda$ is a separate topological space, then we can - obviously - consider \textit{continuous m.s.o.s.m.} (a.k.a. {\em stochastically continuous} \cite{Thang2}), while if $\Lambda$ is a measurable space, i.e. it is endowed with a positive measure $\nu$ on a $\sigma$-algebra $\Sigma$ of subsets of $\Lambda$ ($\Lambda = (\Lambda, \Sigma, \nu)$), we may speak about \textit{measurable m.s.m.} (see \cite[pp. 54]{Thang2}), as well as of (\textit{square}) \textit{summable} or \textit{locally (square) summable m.s.o.s.m.} (see \cite{Mak}, \cite{PGapercorr}).

Also covered by this concept are the {\em stochastic or random measures} (see \cite{Thang4}), \emph{stochastic or random integrals} and the \emph{random operators} (see \cite[pp. 54, 55]{Thang2}). For example, when $\Lambda = H$  and $\Phi$ is a continuous linear operator from $H$ to $ L^{r}_s(\wp, H)$, then it is known as a \emph{random bounded operator} (see \cite{Sko} and, for recent results, \cite{Thang2}, \cite{Thang3}, \cite{Thang4}),
whereas if $\Lambda$ is a dense subspace in $H$, then a closed linear operator from $\Lambda$ to $L^{r}_s(\wp, H)$ is a \emph{random (not necessarily bounded) operator} (see \cite{Hack5}).

The concept of univariate second order stochastic mapping was mentioned as such by H. Niemi in  \cite[pp. 6]{Niemi2}, while the concept of the generalized multivariate second order stochastic mapping was introduced by P. Masani in \cite{MasDilat} under the name \emph{hilbertian variety}. In the framework of this last concept we may place also the works of W. Hackenbroch \cite{Hack4} on Hilbert space operator valued processes, S. A. Chobanyan and A. Weron \cite{ChoWeDiss} on prediction theory in Banach spaces and of I. Suciu and I. Valu\c sescu \cite{SucVal},\cite{Val80}, \cite{Val83} about the study of stochastic processes in the context of complete correlated actions.

An important tool in the development of the theory in all mentioned areas is the reproducing kernel technique for Hilbert spaces and Hilbert $C^*$-modules which are well presented in \cite{BCR}, respectively in \cite{Szaf}.

In what follows we restrain ourselves to the study of \emph{multivariate second order (m.s.o.) stochastic mappings}, which cover not only some particular concepts used in the very well developed theory of \emph{m.s.o. random fields} (see \cite{Kaki}), but also an extension which we have in view.

Namely, if we intent to consider the m.s.o. stochastic processes not only as $\Hs$-valued or $\Bc(L^2_0(\wp), H)$-valued \emph{functions} on $\R^{d}$, but more generally as $\Hs$-valued, respectively $\Bc(L^2_0(\wp), H)$-valued \emph{distributions} on $\mathbb{R}^{d}$ (see \cite{Schw2}), then it is necessary to have $\Lambda=\Dc(\R^{d})=\Dc_{d}$, the space of test functions in the theory of distributions. For such a {(generalized) m.s.o. stochastic mapping} we shall use the term \emph{(generalized) m.s.o. random distribution field (m.s.o.r.d.f.)}.

Mentioning that, for the univariate one time parameter case, such an extension of stochastic processes was considered for the first time by K. It\^{o} (\cite{Ito}) and I.M.Gelfand (\cite{Gelf1}) in 1953 and 1955 respectively and then for the finite variate $d$ time parameteres case in 1957 by A. M. Yaglom \cite{Yaglom57} (see also \cite{FleischKooh}, \cite{GelfVil}, \cite{Yaglom52}, \cite{Yaglom87}, \cite{Yadrenko}), we emphasize that the starting point of our research was the extension of the theory of m.s.o. random fields on $\R^{d}$ (treated in \cite{Kaki}) to the \emph{m.s.o. random distribution fields}.

\section{Correlation theory of multivariate second order stochastic mappings}

In this section we consider \emph{(continuous) m.s.o. stochastic mappings}, indexed over an abstract set (topological space) $\Lambda$, which in what follows will be denoted (in both cases) by $\bM (\Lambda, \Hs)$. In analogy to \cite[Section 4.1]{Kaki} we associate to such stochastic mappings \emph{the vector domain, the modular domain, the measurements space and the operator and scalar (cross) covariance functions}. The corresponding
modular and vector domains are then characterized as reproducing kernel structures (a Hilbert module and a Hilbert space) considered in \cite[Section 2.4]{Kaki}, reproduced by a positive definite $\Cs_1(H)$-valued kernel on $\Lambda$, which is just the operator covariance function of the stochastic mapping. Analogously the measurements space will be identified as a Hilbert space reproduced by the scalar covariance function of the stochastic mapping.
Also for a general positive definite $\Cs_1(H)$-valued kernel, as considered in
\cite[Section 2.4]{Kaki}, a Kolmogorov type factorization in terms of a m.s.o. stochastic mapping is obtained.
Thus, this Section can be regarded as being complementary to Section 2.4. from the book \cite{Kaki}.

For a given m.s.o.s.m. $\Phi\in\bM(\Lambda,\Hs)$ we denote by $\Hs_{(\Phi)}$ the closed linear subspace of $L_{s,0}^{2}(\wp, H)$ generated by the values $\{\Phi(\mu), \mu\in \Lambda\}$ and call it \emph{the vector domain} of $\Phi$, while the \emph{modular domain} of $\Phi$ is the closure $\Hs_{\Phi}$ in $L_{s,0}^{2}(\wp, H)$ of the submodule
\begin{equation}\label{1323}
\Hs^{0}_{\Phi} :\ = \left\{\sum\limits_{i\in\mathbb{N}_{m}} a_{i} \Phi(\lambda_{i})\,,a_{i}\in\Bc(H),\lambda_{i}\in \Lambda, i\in\mathbb{N}_{m}, m\in \mathbb{N}\right\},
\end{equation}
where ${\mathbb N}_m = \{1,2,\dots, m\}$. Obviously, $\Hs_{(\Phi)} \subset  \Hs_{\Phi} \subset \Hs$. These terms can serve also to introduce in $\bM(\Lambda, \Hs)$ the relation of subordination, which will be useful in the general Wold decomposition. Namely, we shall say that $\Phi$ is \emph{subordinate (operator subordinate)} to $\Psi$ if $\Hs_{(\Phi)} \subset \Hs_{(\Psi)}$ ($\Hs_\Phi \subset \Hs_\Psi$, respectively).\\
Now, using the operator model of the normal Hilbert $\Bc(H)$-module $\Hs_{\Phi}$ (see Corr. 7 pp. 30 of \cite{Kaki}) and denoting $G_{\Phi}:=G_{\Hs_{\Phi}}$, which will be called the \emph{measurements space} of $\Phi$, we have the module isomorphisms
\begin{equation}\label{1326}
\Hs_{\Phi}\cong \mathscr{C}_{2}(G_{\Phi},H)=H\widehat{\otimes}_{2}\overline{G}_{\Phi} \cong H \widehat{\otimes}  G_\Phi,
\end{equation}
where $H \widehat{\otimes} G$ is the hilbertian tensor product of the Hilbert spaces $H$ and $G$ as in \cite{Kaki} pp.20.

Also for two elements $\Phi, \Psi \in \bM (\Lambda, \Hs)$ the
\textit{operator cross covariance function} $\Gamma_{\Phi,\Psi}$ will be defined as
\begin{equation}\label{1322}
\Gamma_{\Phi,\Psi}(\lambda,\mu):=\left[\Phi(\lambda),\Psi(\mu)\right]_{\Hs} ;\quad \lambda,\mu\in \Lambda,
\end{equation}
while the \textit{scalar cross covariance function} $\gamma_{\Phi, \Psi}$ is defined by
\begin{equation}\label{eq:corelscal}
\gamma_{\Phi,\Psi}(\lambda,\mu):=\mathrm{tr}\Gamma_{\Phi, \Psi}(\lambda, \mu);\quad \lambda,\mu \in \Lambda.
\end{equation}
When $\Phi = \Psi$, then we denote simply $\gamma_{\Phi, \Phi} =\ :\gamma_{\Phi}$ and $\Gamma_{\Phi, \Phi} =\ : \Gamma_{\Phi}$, which will be called the \emph{scalar covariance function}, respectively the \emph{operator covariance function} of $\Phi$.

Now, the operator covariance function $\Gamma_{\Phi}$ (the scalar one $\gamma_{\Phi}$, respectively) of $\Phi$ is a $\mathscr{C}_{1}(H)$-valued ($\mathbb{C}$-valued) \textit{positive definite kernel} on $\Lambda$, in the sense of the positivity from $\mathcal{B}(H)$, i.e. it holds
\begin{equation}\label{1327}
\sum\limits_{i,j\in \mathbb{N}_{m}}a_{i}\Gamma_{\Phi}(\lambda_{i},\lambda_{j})a_{j}^{\ast}\ge 0,
\end{equation}
for any $m\in\mathbb{N}$ and any finite systems $a_{1},\ldots,a_{m}\in \Bc(H),\  \lambda_{1},\ldots,\lambda_{m} \in \Lambda$, (respectively in the sense of the usual positivity in $\mathbb{C}$)
\begin{equation}\label{1328}
\sum\limits_{i,j\in \mathbb{N}_{m}}\alpha_{i}\overline{\alpha}_{j}\gamma_{\Phi}(\lambda_{i},\lambda_{j})\geq 0
\end{equation}
for any $m \in \mathbb{N}$ and any finite systems $\alpha_{1},\ldots,\alpha_{m}\in \mathbb{C}, \lambda_{1},\ldots,\lambda_{m}\in \Lambda)$.

Indeed the relation \eqref{1327} results by applying (\ref{1322}) and computing:

\begin{multline*}
\sum\limits_{i,j\in \mathbb{N}_{m}}a_{i}\Gamma_{\Phi}(\lambda_{i},\lambda_{j})a_{j}^{\ast}=\sum\limits_{i,j\in \mathbb{N}_{m}}a_{i}\left[\Phi(\lambda_{i}),\Phi (\lambda_{j})\right]_{\Hs}a_{j}^{\ast}=\\
\sum\limits_{i,j\in \mathbb{N}_{m}}\left[a_{i}\Phi(\lambda_{i}),a_{j}\Phi(\lambda_{j})\right]_{\Hs}=\left[\sum\limits_{i\in \mathbb{N}_{m}}a_{i}\Phi(\lambda_{i}),\sum\limits_{j\in \mathbb{N}_{m}}a_{j}\Phi(\lambda_{j})\right]_{\Hs}\geqslant 0,
\end{multline*}
while \eqref{1328} appears also in \cite{Niemi2}.

Note also that the positive definite $\mathscr{C}_{1}(H)$~-~valued kernel $\Gamma=\Gamma_{\Phi}$ reproduces a normal Hilbert  $\Bc(H)$-module $\Hs_{\Gamma_{\Phi}}$ (see \cite{Kaki} Section 2.4, Thm. 13, pp.37), as well as a Hilbert space $G_{\Gamma_{\Phi}}$ of $H$-valued functions on $\Lambda$ (see \cite{Kaki}, Prop. 23, Section 2.4, pp.44), while $\gamma=\gamma_{\Phi}$ as complex valued positive definite kernel reproduces a Hilbert space $K_{\gamma_{\Phi}}$, (see for example \cite{BCR}).\\
It is not hard to infer now, extending the correspondence $\Phi(\lambda) \mapsto \Gamma_{\Phi}(\lambda,\cdot)$ to a $\Bc(H)$-linear mapping between the generating submodules of $\Hs_{\Phi}$ and $\Hs_{\Gamma_{\Phi}}$, that these are isomorph (i.e. $\Hs_{\Phi}\approxeq \Hs_{\Gamma_{\Phi}}$) as normal Hilbert $\Bc(H)$-modules and also that the measurements space $G_{\Phi}$ associated to $\Phi$ is isomorph to the Hilbert space $G_{\Gamma_{\Phi}}$ ($G_{\Phi}\simeq G_{\Gamma_{\Phi}}$) having $\Gamma_{\Phi}$ as operator reproducing kernel.

Moreover extending the correspondence
$\Phi(\lambda)\mapsto \gamma_{\Phi}(\lambda,\cdot)$
to a linear mapping between generating subspaces of $\Hs_{(\Phi)}$ and of $K_{\gamma_{\Phi}}$, respectively, we have that these are also isomorph as Hilbert spaces.

On the other hand, similarly to (\ref{1326}) the module isomorphisms

\begin{equation}\label{1335}
\Hs_{\Gamma_{\Phi}}\cong \mathscr{C}_{2}\left(G_{\Gamma_{\Phi}},H\right)\cong H\widehat{\otimes}{G}_{\Gamma_{\Phi}}
\end{equation}
also hold.

Noticing that $G_{\Gamma_{\Phi}}$ is generated by elements of the form $\Gamma_{\Phi}(\cdot,\lambda)x$ with $\lambda\in \Lambda, x\in H$, as well as the fact that the normal Hilbert $\Bc(H)$-module isomorphism $\Hs_{\Phi}\approxeq \Hs_{\Gamma_{\Phi}}$ was constructed by the $\Bc(H)$-linear extension of the correspondence
$
\Phi(\lambda)\mapsto\Gamma_{\Phi}(\lambda,\cdot),
$
the measurements space $G_{\Phi}$ associated to the m.s.o. stochastic mapping  $\Phi$ coincides with the closed subspace in $L^{2}_{0}(\wp)$ generated by

\begin{equation}\label{1337}
\left\{V^*_{\Phi(\lambda)} x \, , \, \lambda\in \Lambda, \ x\in H\right\},
\end{equation}
where $V$ is the module isomorphism \eqref{eq:coresp.weak.rnd.var}.

The above results can be gathered in

\begin{teo}\label{3afirmcorel}
Given an arbitrary set (a topological space) $\Lambda$ and a (continuous) m.s.o. stochastic mapping
\(\Phi: \Lambda\rightarrow \Hs\)
the following assertions hold:
\begin{enumerate}
\item[(i)] the scalar covariance function $\gamma_{\Phi}$, the operator covariance function $\Gamma_{\Phi}$ respectively, is a complex, a $\mathscr{C}_{1}(H)$-valued respectively, positive definite (and continuous) kernel on $\Lambda$;
\item[(ii)] the space $\Hs_{(\Phi)}$ (the module $\Hs_{\Phi}$ respectively), associated to $\Phi$ is isomorph as Hilbert space, (as normal Hilbert  $\Bc(H)$-module respectively)  to the reproducing kernel Hilbert space $K_{\gamma_{\Phi}}$, (the reproducing kernel normal Hilbert $\Bc(H)$-module $\Hs_{\Gamma_{\Phi}}$ respectively);
\item[(iii)] the measurements space $G_{\Phi}$, associated to $\Phi$ is isomorph to the Hilbert space $G_{\Gamma_{\Phi}}$, having $\Gamma_{\Phi}$ as operatorial reproducing kernel. In this way the two Hilbert spaces ensure respectively the description of the $\Bc(H)$-modules $\Hs_{\Phi}, \Hs_{\Gamma_{\Phi}}$ by means of the operatorial models mentioned in (\ref{1326}), respectively (\ref{1335}).\\
    Moreover, $G_\Phi$ can be described directly in terms of the m.s.o. stochastic mapping $\Phi$ by \eqref{1337}.
\end{enumerate}
\end{teo}

The {\it proof} being contained in the considerations preceding the theorem we shall only mention that, in the case where $\Lambda$ is a topological space, the continuity of $\Phi$, easily implies the continuity of $\Gamma_{\Phi}$, which in fact is equivalent to the continuity on the diagonal of $\Lambda\times\Lambda$.
\hfill \qedsymbol

The reproducing kernel technique for normal $\Bc(H)$-modules will be also used to characterize the subordination of two m.s.o.s.m. in terms of their operator (cross) covariance functions. Namely, in analogy to the characterization given in \cite{Kaki} we have

\begin{teo}\label{thm:subordination}
If $\Phi$ and $\Psi$ are two m.s.o.s.m. having the operator covariance functions $\Gamma_{\Phi}$ and $\Gamma_{\Psi}$ respectively, and the operator cross covariance function $\Gamma_{\Phi, \Psi}$, then $\Phi$ is subordinate to $\Psi$, iff the function $K_\lambda(\cdot)=\Gamma_{\Phi, \Psi}(\lambda,\cdot)$ is in the $\Bc(H)$-module $\Hs_{\Gamma_{\Psi}}$ with $\Gamma_{\Psi}$ as reproducing kernel and
\begin{equation}
[K_{\lambda},K_{\mu}]_{\Hs_{\Gamma_{\Psi}}}=\Gamma_{\Phi}(\lambda,\mu)\,;\quad\lambda,\,\mu \in\Lambda.
\end{equation}
\end{teo}
The {\it proof} is based on the fact that the operator cross covariance function $\Gamma_{\Phi, \Psi}$ of $\Phi, \Psi \in \bM(\Lambda, \Hs)$ appears in the expression of the operator covariance function of the product $\Phi\times \Psi \in \bM(\Lambda, \Hs\times\Hs)$ ($\Hs \times \Hs$ being the ``product'' $\Bc(H)$-module as in \cite[Section 4.9, pp. 192]{Kaki}), namely
\[
\Gamma_{\Phi\times\Psi}(\lambda, \mu) = \begin{pmatrix} \Gamma_\Phi (\lambda, \mu) & \Gamma_{\Phi, \Psi} (\lambda, \mu) \\ \Gamma_{\Psi, \Phi} (\lambda, \mu) & \Gamma_\Psi (\lambda, \mu) \end{pmatrix} ,\quad \lambda, \mu \in \Lambda.
\]
Having in view the preceding results, the rest of the proof runs similarly as in Theorem 9.3, pp. 193 from \cite[Chap.4]{Kaki}. \hfill \qedsymbol

If we consider the operator correlation mapping

\begin{equation}\label{preimage}
\mathbf{M}(\Lambda, \Hs)\ni\Phi \mapsto \Gamma_{\Phi}\in \mathbf{\Gamma}\left(\Lambda,\mathscr{C}_{1}(H)\right),
\end{equation}
from the set of m.s.o.s.m. $\mathbf{M}(\Lambda, \Hs)$ to the set of all positive definite $\Cs_1(H)$-valued kernels on $\Lambda$ (see \eqref{1327}), $\mathbf{\Gamma}\left(\Lambda,\mathscr{C}_{1}(H)\right)$, then the next theorem proves its
surjectivity and gives a description of the pre-image of each $\Gamma\in\mathbf{\Gamma}\left(\Lambda,\mathscr{C}_{1}(H)\right)$.

\begin{teo}\label{3afirmASM}
\begin{enumerate}
\item[(i)] If $\Lambda$ is an arbitrary set (a topological space), then for any positive definite $\mathscr{C}_{1}(H)$~-~valued kernel $\Gamma$ on $\Lambda$ (continuous on $\Lambda\times \Lambda$), there exists a (continuous) multivariate second order stochastic mapping $\Phi: \Lambda\rightarrow \Hs$, which assures for $\Gamma$ a Kolmogorov factorization:
    \begin{equation}\label{eq:Kolm.fact}
    \Gamma(\lambda, \mu)=[\Phi(\lambda), \Phi(\mu)]_{\Hs}\ ;\ \lambda, \mu \in \Lambda,
    \end{equation}
    i.e. $\Gamma$ coincides with the operator covariance function $\Gamma_{\Phi}$ of $\Phi$.
\item[(ii)] If $\Phi_{1}$ and $\Phi_{2}$ are two (continuous) m.s.o. stochastic mappings on $\Lambda$ for which $\Gamma_{\Phi_{1}}=\Gamma_{\Phi_{2}}$, then $\Phi_{1}$ and $\Phi_{2}$ are gramian unitary equivalent, i.e. there exists a gramian unitary operator $\mathrm{W}:\Hs_{\Phi_{2}}\rightarrow \Hs_{\Phi_{1}}$ such that
\end{enumerate}
\begin{equation}\label{1339}
\Phi_{1}(\lambda)= \mathrm{W} \Phi_{2}(\lambda)\, , \quad \lambda\in \Lambda.
\end{equation}
\end{teo}
\begin{proof}
(i) Let $\Gamma$ be a positive definite ${\mathscr C}_1(H)$-valued kernel on $\Lambda$ and $\Hs_\Gamma$ ($G_\Gamma$ respectively) the normal Hilbert $\Bc(H)$-module (Hilbert space respectively) reproduced by $\Gamma$, generated as in \cite{Kaki} Sec.2.4,Thm.13,pp.37 or Prop.23,pp.44, with the gramian (respectively the scalar product) defined there. Since, by Thm. 2.4., pp. 45 of \cite{Kaki}, ${\mathscr C}_2(G_\Gamma, H)$ is an operatorial model for $\Hs_\Gamma$, embedding $G_\Gamma$ into a space of the form $L^2_0(\wp)$ (see \cite{Rao1}) $\Hs_{\Gamma}$ embeds into a normal Hilbert $\Bc(H)$-module of the form $\mathscr{C}_{2}\left(L^{2}_{0}(\wp),H\right)$. \\
In this way the $\Hs_{\Gamma}$-valued function defined by
   \( \Phi(\lambda):=\Gamma(\lambda,\cdot),\  (\lambda\in \Lambda) \)
gives us exactly the m.s.o. stochastic mapping we seek, since we get for any $\lambda,\mu \in \Lambda$
    \begin{equation*}
    \Gamma_{\Phi}(\lambda,\mu)=\left[\Phi(\lambda),\Phi(\mu)\right]_{\Hs}=\left[\Gamma(\lambda,\cdot),
    \Gamma(\mu,\cdot)\right]_{\Hs_{\Gamma}}=\Gamma(\lambda,\mu).
    \end{equation*}
(ii) From the hypothesis on $\Phi_1$ and $\Phi_2$ we have
    \begin{equation*}
    \left[\Phi_{1}(\lambda),\Phi_{1}(\mu)\right]_{\Hs}=\left[\Phi_{2}(\lambda),\Phi_{2}(\mu)\right]_{\Hs},\quad \lambda,\mu \in \Lambda,
    \end{equation*}
    which means that the mapping $\mathrm{W}$ defined on $\left\{\Phi_{2}(\mu), \mu\in \Lambda\right\}$ by

    \begin{equation*}
    \mathrm{W}\Phi_{2}(\mu)=\Phi_{1}(\mu)\, , \quad \mu\in \Lambda
    \end{equation*}
    preserves the gramian, is easily extended by $\Bc(H)$-linearity, still conserving the gramian and is clearly surjective between the $\Bc(H)$-modules $\Hs_{\Phi_{2}}^{0}$ and $\Hs_{\Phi_{1}}^{0}$ (see (\ref{1323})). This allows the extension of ${\mathrm W}$ by continuity to a gramian unitary operator from $\Hs_{\Phi_{2}}$ to $\Hs_{\Phi_{1}}$, satisfying (\ref{1339}).
    \end{proof}

\begin{rem}
Analogous results, formulated in a Hilbert space setting, hold for the scalar correlation mapping
\begin{equation}
\bM (\Lambda,\Hs)\ni\Phi\mapsto\gamma_{\Phi}\in \bgam(\Lambda),
\end{equation}
where $\bgam(\Lambda)$ is the set of complex valued positive definite kernels on $\Lambda$.
\end{rem}

\section{Multivariate second order random distribution fields and their covariance distributions}

In this Section we illustrate how the results from the previous Section can be applied to the case of multivariate second order (m.s.o.) random distribution fields.

For completeness we shall work with the \emph{$m$-order $(0\leqslant m\leqslant\infty)$ m.s.o. random distribution field} $U$, which is a continuous linear m.s.o stochastic mapping having the index set $\Lambda=\Dc_{d}^{m}$, the space of complex valued compactly supported continuously derivable  functions up to the order $m\, (m=0,1,2,\ldots,\infty)$. In other words $U\in \big(\Dc^{m}_{d}\big)'(\Hs):=\Bc(\Dc_{d}^{m}, \Hs)$, i.e. it is a $\Hs$-valued $m$-order distribution on $\mathbb{R}^{d}$ (see also \cite{Schw2}). For $m=0$, $U$ will be a ($\Hs$-valued) m.s.o. random Radon measure and for $m=\infty$, the index $m$ and the ``$m$-order''  will be omitted, otherwise $U$ will be  called of \emph{finite $m$-order}.

Now for $U, V\in \big(\Dc^{m}_{d}\big)'(\Hs)$ the scalar cross covariance function $\gamma_{U,V}$ and the operator cross covariance function $\Gamma_{U,V}$ are defined by (\ref{eq:corelscal}) and \eqref{1322}.
We shall comment only the properties of $\Gamma_{U,V}$, which -- from the linearity and continuity of $U$ and $V$ -- is sesquilinear and continuous on $\Dc_d^m\times \Dc_d^m$. Moreover $\Gamma_U$ is obviously positive definite, i.e. it satisfies a relation of the form (\ref{1327}).

It is possible as in Theorems \ref{3afirmcorel} and \ref{3afirmASM} to describe and to characterize the m.s.o. random distribution fields, their modular and vector domains and measurements space in terms of such kernels. We prefer but to do that by using everywhere a distributional framework, where instead of $\Cs_1(H)$-valued sesquilinear continuous kernels on $\Dc_d^m\times\Dc_d^m$, we use the so-called distribution kernels on $\R^d$ in the sense of L. Schwartz (see \cite[I, pp. 138]{Schw1}), which are $m$~-~order distributions on $\mathbb{R}^{2d}$.
Therefore we need the following Lemma, where we identify $\Dc_{2d}^m$ with the inductive tensor product $\Dc_d^m\otimes_i \Dc_d^m$ (see for example \cite[pp. 84]{Groth}).
\begin{lem}\label{lem:trei.unu}
Let $\Gamma$ be a $\Cs_1(H)$-valued sesquilinear kernel on $\Dc_d^m$, which is continuous on the diagonal of $\Dc_d^m\times\Dc_d^m$. Then there is an $m$-order distribution kernel $C_\Gamma$ on $\R^d$ such that
\begin{equation}
C_\Gamma (\varphi\otimes\psi) = \Gamma(\varphi,\bar\psi),\quad \varphi,\psi\in\Dc_d^m.
\end{equation}
Moreover, when $\Gamma$ is positive on the diagonal, then $C_\Gamma$ is a positive definite $m$-order distribution kernel.
\end{lem}
\begin{proof}
First, from the hypothesis, it is easily seen that the kernel $\Gamma$ is continuous on $\Dc_d^m \times \Dc_d^m$ and when $\Gamma(\varphi, \varphi) \ge 0,\ \varphi \in \Dc_d^m$, then $\Gamma$ is also positive definite. We shall attach an $m$-order $\mathscr{C}_{1}(H)$~-~valued distribution $C_{\Gamma}$ on $\mathbb{R}^{2d}$ as follows. Define first $C_{\Gamma}$ on the elementary tensors $\varphi\otimes\psi$ from $\Dc^{m}_{d}\otimes\Dc^{m}_{d}$ by
$C_{\Gamma}(\varphi\otimes\psi):= \Gamma(\varphi,\overline{\psi})$,
extend that by linearity, and then
by continuity to the whole $\Dc^{m}_{d}\widehat{\otimes}_{i}\Dc^{m}_{d}=\Dc^{m}_{2d}$ preserving the notation $C_{\Gamma}$. \\
So $C_{\Gamma}\in\big(\Dc^{m}_{2d}\big)'(\mathscr{C}_{1}(H))$, which, when $\Gamma$ is positive on the diagonal, will be a positive definite distribution kernel on $\R^d$. We shall refer to this as $C_{\Gamma}\in pd(\Dc^{m}_{2d})'\big(\Cs_{1}(H)\big)$.
\end{proof}

Since, for each $U,V\in \big(\Dc^{m}_{d}\big)'(\Hs)$ the operator cross covariance functions $\Gamma_{U,V}$, $\Gamma_U$ satisfy the hypotheses from Lemma \ref{lem:trei.unu}, the existence of the distribution kernels $C_{\Gamma_{U,V}}$ and $C_{\Gamma_{U}}$ is assured, the last one being even positive definite. These will be called the \emph{operator cross covariance distribution} of $U$ and $V$, respectively the \emph{operator covariance distribution of} the m.s.o. random distribution field \emph{$U$} and denoted by $C_{U,V}$, respectively $C_U$. Correspondingly $c_{U,V}$ defined by $c_{U,V}(\chi):=tr C_{U,V}(\chi), \, \chi\in \Dc^{m}_{2d}$, will be called the \emph{scalar cross covariance distribution} of $U$ and $V$, respectively $c_U = c_{U,U}$ the \emph{scalar covariance distribution} of $U$.\\
By applying Theorem \ref{3afirmcorel}, the following description of the domains associated to a m.s.o.r.d.f. in terms of its covariance distribution holds.
\begin{cor}\label{cor3.2}
For a given m.s.o. random distribution field $U\in\big(\Dc^{m}_{d}\big)'(\Hs)$, the modular domain $\Hs_{U}$, respectively the measurements space $G_{U}$ will be isomorph to the $\Bc(H)$-module $\Hs_{\Gamma^{C_{U}}}$, respectively  the Hilbert space $G_{\Gamma^{C_{U}}}$ both reproduced by the $\mathscr{C}_{1}(H)$-valued kernel $\Gamma^{C_{U}}$, while the vector domain $\Hs_{(U)}$ of $U$ is isomorphic to the  Hilbert space reproduced by the complex valued kernel $\gamma^{C_{U}}$.
\end{cor}
The operator cross covariance distribution together with the operator covariance distributions of two m.s.o.r.d.f. will be used to obtain from Theorem \ref{thm:subordination} the following characterization of their subordination. \begin{cor}
Given two m.s.o.r.d.f. $U$ and $V$ then $U$ is subordinate to $V$, iff for each fixed $\varphi \in \Dc_d$, the function $K_\varphi (\cdot)$ given by
\[
\Dc_d \ni \psi \mapsto C_{U,V} (\varphi \otimes \bar \psi) \in \Cs_1(H)
\]
lies in the $\Bc(H)$-module $\Hs_{\Gamma^{C_V}}$ and $\bigl[K_\varphi , K_\psi\bigr]_{\Hs_{\Gamma^{C_V}}} = C_U (\varphi \otimes \bar \psi)$.
\end{cor}
The operator correlation mapping (\ref{preimage}) becomes
\begin{equation}\label{3.5}
(\Dc_{d}^{m})'(\Hs)\ni U \mapsto C_{U} \in pd (\Dc_{2d}^{m})'(\Cs_{1}(H)),
\end{equation}
which we shall call the \emph{operator covariance distribution mapping}.
Its properties are contained in the following Corollaries.

\begin{cor}\label{cor3.1}
For an $m$-order distribution kernel on $\R^d$, $C\in\big(\Dc^{m}_{2d}\big)'(\mathscr{C}_{1}(H))$ the following statements are equivalent:
\begin{itemize}
\item[(i)] $C\in pd\big(\Dc^{m}_{2d}\big)'(\mathscr{C}_{1}(H))$; i.e. the operatorial kernel $\Gamma^{C}$ on $\Dc^{m}_{d}$ defined by
\(
\Gamma^{C}(\varphi,\psi):=C(\varphi\otimes\overline{\psi})\ (\varphi,\psi\in \Dc^{m}_{d})
\)
is positive definite.
\item[(ii)] The operatorial kernel $\Gamma^{C}$ is sesquilinear and positive (in the sense of positivity from $\Bc(H)$) on the diagonal of $\Dc^{m}_{d}\times\Dc^{m}_{d}$.
\item[(iii)] There is a m.s.o. random distribution field $U\in\big(\Dc^{m}_{d}\big)'(\mathscr{H})$  such that $C=C_{U}$ (i.e. $C$ has a Kolmogorov type factorization).
\end{itemize}
\end{cor}

The uniqueness (up to a gramian unitary equivalence) of $U$ in (iii) of the previous Corollary results now from Theorem \ref{3afirmASM}(ii). More precisely it holds

\begin{cor}\label{cor3.3}
If the m.s.o. random distribution fields $U^{1},\,U^{2}\in \big(\Dc^{m}_{d}\big)'(\Hs)$ have the same operator covariance distribution, i.e. $C_{U^{1}}=C_{U^{2}}$, then there exists a gramian unitary operator $W:\Hs_{U^{1}}\rightarrow \Hs_{U^{2}}$ such that
\begin{equation}
W\,U^{1}_{\varphi}=U^{2}_{\varphi},\quad \varphi\in \Dc^{m}_{d}.
\end{equation}
\end{cor}

It would be of course interesting to know \emph{if each $\Cs_1(H)$-valued distribution kernel on $\R^d$ is the operator cross covariance distribution of some pair $U,V \in (\Dc_d^m)'(\Hs)$}, but this won't be our goal here. \\
Now we shall introduce in our general frame the concepts of determinism and nondeterminism and we give a general decomposition of Wold type of a m.s.o.r.d.f., $U$, into deterministic and purely nondeterministic parts, both summands being operator subordinate to $U$.

In \cite{Roz2} (see also \cite{Bala}), the observable space up to the moment $t_{0}$ for a random distribution for the case $d=1$ and $H=\mathbb{C}$ was defined. By using in $\R^d$ the ordering relation $s=(s_1, \dots , s_d) \le t = (t_1, \dots , t_d)$ if $s_j \le t_j,\ j=1,\dots,d$, we shall define by analogy these observable structures for m.s.o.r.d.f. up to the moment $t_0$.\\
To this purpose we introduce first the subspace $\Dc_{d}^{t_{0}}$ in $\Dc_{d}$ by
\begin{equation}
\varphi\in\Dc_{d}^{t_{0}}\quad\Leftrightarrow\quad \textrm{supp} \varphi \subset \left\{t\in \mathbb{R}^{d}\,:\,t\leqslant t_{0}\right\}.
\end{equation}
Thus, for a m.s.o.r.d.f. $U = \{ U_\varphi \}_{\varphi\in\Dc_d^m}$ we call \textit{the observable module} $\Hs_{U}^{t_{0}}$ (\textit{the observable space} $\Hs_{(U)}^{t_{0}}$) \textit{up to the moment $t_{0}$}, the closed $\Bc(H)$-module (space) in $\Hs$, generated by the set
$\left\{U_{\varphi}\quad,\, \varphi\in \Dc_{d}^{t_{0}}\right\}$.

We notice that the Hilbert $\Bc(H)$-module (Hilbert space) generated by $\bigcup\limits_{t\in \mathbb{R}^{d}}\Hs_{U}^{t}$ ($\bigcup\limits_{t\in \mathbb{R}^{d}}\Hs_{(U)}^{t}$ respectively) is just the modular domain $\Hs_{U}$ (vector domain respectively) of the m.s.o.r.d.f. $U$, while $\bigcap\limits_{t\in \mathbb{R}^{d}}\Hs_{U}^{t}$ ($\bigcap\limits_{t\in \mathbb{R}^{d}}\Hs_{(U)}^{t}$ respectively) will be called the \emph{remote past} module (space) of $U$ and denoted by $\Hs_{U}^{-\infty}$ ($\Hs_{(U)}^{-\infty}$ respectively). We shall say that $\varphi\mapsto U_{\varphi}$ is \textit{operator deterministic} (or, simply \textit{deterministic}), if $\Hs_{U}=\Hs_{U}^{-\infty}$ ($\Hs_{(U)}=\Hs_{(U)}^{-\infty}$ respectively). If  $\Hs_{U}^{-\infty}\subsetneqq \Hs_{U}$ ($\Hs_{(U)}^{-\infty}\subsetneqq \Hs_{(U)}$ respectively) we say that $U$ is \textit{operator nondeterministic} (\textit{nondeterministic}).

For extreme situations, where $\Hs_{U}^{-\infty}=\{0\}$ , ($\Hs_{(U)}^{-\infty}=\{0\}$ respectively), the term \emph{non-deterministic} becomes \emph{purely non-deterministic} (see also \cite{Wold}, \cite[Sec. 5.1]{Kaki} for the first, respectively the classical result). In this context we have the following general Wold decomposition for m.s.o.r.d.f.

\begin{teo}\label{thm:WoldRDF}
For a m.s.o.r.d.f. $\{U_{\varphi}\}_{\varphi\in \Dc_d}$ taking values in $\Hs$, there exists a unique decomposition

\begin{equation}\label{426}
U_{\varphi}=U_{\varphi}^{det}+U_{\varphi}^{p}\,,\quad \varphi\in \Dc_d
\end{equation}
of $U$ such that
\begin{enumerate}
\item[(i)]$\{U_{\varphi}^{det}\}_{\varphi\in \Dc_d}$ is operator deterministic and  $\{U_{\varphi}^{p}\}_{\varphi\in \Dc_d}$ is operator purely nondeterministic,
\item[(ii)] $U^{det}$ and $U^{p}$ are operator subordinate to $U$,
\item[(iii)] $U^{det}$ and $U^{p}$ are operator uncorrelated, i.e. the operator cross covariance distribution $C_{U^{det}, U^p}$ vanishes on $\Dc_{2d}$.
\end{enumerate}
Moreover the gramian orthogonal decomposition
\begin{equation*}
\Hs_{U}=\Hs_{U^{det}}\oplus \Hs_{U^{p}}
\end{equation*}
holds.
\end{teo}
\begin{proof}
Since in normal Hilbert $\Bc(H)$-modules, any closed submodule has a gramian projection (Lemma 2, Section 2.2, pp.22 \cite{Kaki}), let $P$ be the gramian projection associated to the submodule $\Hs_{U}^{-\infty}$ and let
\begin{equation}\label{427}
U_{\varphi}^{det}=PU_{\varphi}\quad,\quad U_{\varphi}^{p}=(I-P)U_{\varphi}\,,\quad \varphi\in \Dc_d.
\end{equation}
One can verify by using standard arguments (as in \cite[Section 5.1]{Kaki}) that \eqref{427} gives us the very m.s.o.r.d.f. we seek, the uniqueness of decomposition \eqref{426} being obtained by direct reasoning.
\end{proof}

\section{Multivariate second order stochastic measures and their associated bimeasures}

In this Section we shall see how from the theory of m.s.o.r.d.f. we can indeed infer the classical theory of m.s.o. random fields, as well as how to fit in this general theory the special class of m.s.o.r.d.f., namely that of (not necessarily bounded) measures, which is very useful for the spectral representations of random (distribution) fields. \\
Each very well known continuous m.s.o. random field $F$ on $\mathbb{R}^{d}$ (i.e. $F\in\mathcal{E}^{0}_{d}(\Hs)$) will be identified with the m.s.o. random Radon measure $U^{F}$ (i.e. m.s.o. random distribution fields of zero order having as index set $\Dc_d^0$), defined by
\begin{equation}\label{4.1}
U_{\varphi}^{F}=\int\limits_{\mathbb{R}^{d}}\varphi(t) F(t) dt,\quad \, \varphi\in \Dc^0_{d}.
\end{equation}
Moreover, since $F$ is bounded on each bounded Borel subset $A$ of $\R^d$ (i.e. on each $A \in \widetilde\Bc or(\R^d)$), it can be even regarded as a not necessarily bounded regular measure $\xi^{F}$ through
\begin{equation}\label{3.2}
\xi^{F}(A)=\int\limits_{A}F(t)dt,\quad A\in\widetilde{\mathcal{B}}or(\mathbb{R}^{d}).
\end{equation}
In such a way the locally convex $\Bc(H)$-module $\Ec_d^0(\Hs)$ is embedded into the class of \emph{multivariate second order stochastic, not necessarily bounded, regular measures}, for which the index set $\Lambda$ is the $\delta$-ring $\widetilde{\Bc}or(\mathbb{R}^{d})$. Playing an important role in the definitions of various kinds of harmonizability of m.s.o. random distribution fields, they are often supposed to have finite operator semivariation i.e.
\begin{equation}\label{eq:semivar.op}
\|\xi\|_{o}(A)< \infty,\quad A\in \widetilde{\Bc}or(\mathbb{R}^{d}),
\end{equation}
where $\|\cdot\|_{o}$ is defined as in \cite[pp. 56 Def.4 (1)]{Kaki}, or finite semivariation i.e.
\begin{equation}\label{eq:semivar}
\|\xi\|(A)< \infty,\quad A\in \widetilde{\Bc}or (\mathbb{R}^{d}),
\end{equation}
where $\|\cdot\|$ is defined as in \cite[pp.54 Def.1]{Kaki}, the corresponding classes of m.s.o. stochastic regular measures being denoted by $fosvr\mathcal{M}_{d}(\Hs)$, or by $fsvr \mathcal{M}_{d}(\Hs)$, respectively. They are locally convex $\Bc(H)$-modules with the seminorms given by \eqref{eq:semivar.op} and \eqref{eq:semivar} respectively.
However, since $\|\xi\|(A)\leqslant\|\xi\|_{0}(A),\, A\in\widetilde{\Bc}or(\mathbb{R}^{d})$, the inclusion $fosvr\mathcal{M}_{d}(\Hs) \subset \linebreak fsvr \mathcal{M}_{d}(\Hs)$ holds with continuous embedding.

Now an element $\xi\in fsvr\,{\mathcal M}_{d}(\Hs)$ is to be regarded as a $\Hs$-valued distribution $U^{\xi}$, which will be even a $\Hs$~-~valued Radon measure, by putting

\begin{equation}\label{3.3}
U_{\varphi}^{\xi}=\int\limits_{\mathbb{R}^{d}}\varphi(t)d\xi(t),\quad \varphi\in\Dc_{d}^0.
\end{equation}
Thus, we have obtained
\begin{prop}\label{prop:incluziuni}
For the above mentioned classes of m.s.o. random distribution fields the following inclusions hold with continuous embedding
\begin{multline}\label{3.4}
\Ec^{0}_{d}(\Hs)\subset fosvr\,{\mathcal M}_{d}(\Hs)\subset fsvr\,{\mathcal M}_{d}(\Hs)\subset\\
\subset(\Dc_{d}^{0})'(\Hs)\subset(\Dc_{d}^{m})'(\Hs)\subset\Dc'_{d}(\Hs).
\end{multline}
\end{prop}
Let's observe that a m.s.o. stochastic mapping from one of the mentioned classes can have more than one index set. For example, looking to the above identifications:
\[
\xi = \{ \xi(A)\}_{A\in \widetilde\Bs or (\R^d)}\ \text{ with } \ U^\xi=\{ U^\xi_\varphi\}_{\varphi\in\Dc_d^0},
\]
respectively
\[
F = \{ F(t)\}_{t\in\R^d} \ \text{ with }\ \xi^F = \{\xi^F(A)\}_{A\in\widetilde\Bs or(\R^d)} \ \text{ and } \ U^F = U^{\xi^F}
\]
it is not difficult to see that
\[
\Hs_\xi = \Hs_{U^\xi},\ \text{ respectively }\ \Hs_F = \Hs_{\xi^F} = \Hs_{U^F},
\]
analogous relations being true for the corresponding vector domains and measurements spaces.
It is of interest how those classes of positive definite kernels which are co-domains of the ``restrictions'' of the operator covariance distribution mapping \eqref{3.5} to the submodules from \eqref{3.4} can be suitable described.
Let's mention that for the operator covariance distribution of an element $\xi \in fsvr{\mathcal M}_d(\Hs)$ regarded as a m.s.o. random distribution field we naturally use the notation $\Gamma_{U^\xi}$, while if $\xi$ is regarded as a m.s.o. stochastic mapping on
$\widetilde\Bc or (\R^d)$, then its operator covariance function $\Gamma_\xi$ represents a positive definite regular bimeasure on $ \widetilde\Bc or (\R^d) \times \widetilde\Bc or (\R^d)$, for which we shall use also the notation $\tau_\xi$. It will be also called the \emph{operator covariance bimeasure} associated to the m.s.o stochastic measure $\xi$.

For the bimeasures on $ \widetilde\Bc or (\R^d) \times \widetilde\Bc or (\R^d)$ and their semivariation or operator semivariation we adopt analogue definitions as in \cite{Kaki} (Definition 9 (1) and (3) pp.62 and Definition 16 pp.65). For the spaces of such ($\Cs_{1}(H)$-valued) regular bimeasures with finite (operator) semivariation we shall use the notation $fsvr \, \mathfrak{M}_{2d}\big(\Cs_{1}(H)\big)$ ($fosvr \,\mathfrak{M}_{2d}\big(\Cs_{1}(H)\big)$, respectively).\\
However, since in our case, the measure and the bimeasure are defined on a $\delta$-ring, some properties from \cite{Kaki} do not automatically hold.

The corresponding classes of positive definite bimeasures will be denoted by $fsvr \, \mathfrak{M}_{2d}^{pd}\big(\Cs_{1}(H)\big)$ and $fosvr \, \mathfrak{M}_{2d}^{pd}\big(\Cs_{1}(H)\big)$. So in the above notation $\tau_\xi \in f(o)svr{\mathfrak M}_{2d}^{pd}(\Cs_1(H))$.\\
Let's also mention that it is not hard to see that a Morse-Transue (MT)\emph{strict integral} as in \cite[Section 1.2 pp.5]{Kaki} can be defined for bimeasures $\tau$ on $\widetilde\Bc or (\R^d) \times \widetilde \Bc or (\R^d)$. With such a strict MT-integral to each $\tau \in fsvr \Mg_{2d}(\Cs_1(H))$ we can attach a distribution $C^{\tau}$ on $\R^{2d}$ (a distribution kernel on $\R^d$, in the sense of L. Schwartz), first defined on elementary tensors trough

\begin{equation}\label{eq:covdistrib}
C^{\tau}(\varphi\otimes\psi)=\int\limits_{\mathbb{R}^{d}}\int\limits_{\mathbb{R}^{d}}
\varphi(s)\psi(t)d\tau(s,t),\quad \varphi,\psi\in \Dc_{d}^{m},
\end{equation}
and then, by the usual extension, to the whole $\Dc^{m}_{2d}$.

On the other hand, each $K\in \Ec^{0}_{2d}\big(\Cs_{1}(H)\big)$ can be regarded as a regular bimeasure $\tau^{K}$, having finite operator semivariation, by putting
\begin{equation}
\tau^{K}(A,B):=\int\limits_{A}\int\limits_{B} K(s,t)ds dt,\quad A, B\in\widetilde{\Bc}or(\mathbb{R}^{d}).
\end{equation}

This infers
\begin{prop}\label{prop:incluz.poz.defn}
The valued domains of the restrictions of the operator covariance distribution mapping \eqref{3.5} to the spaces from \eqref{3.4}
satisfy respectively the following inclusions
\begin{multline}\label{eq:incluz.pos.def}
pd\Ec^{0}_{2d}\big(\Cs_{1}(H)\big)\subset fosvr\,\mathfrak{M}_{2d}^{pd}\big(\Cs_{1}(H)\big)\subset fsvr\,\mathfrak{M}_{2d}^{pd}\big(\Cs_{1}(H)\big)\subset\\\subset pd(\Dc^{0}_{2d})'\big(\Cs_{1}(H)\big)
\subset pd(\Dc^{m}_{2d})'\big(\Cs_{1}(H)\big)\subset pd(\Dc_{2d})'\big(\Cs_{1}(H)\big).
\end{multline}
\end{prop}
For a more complete image we discuss in detail the restrictions of the operator covariance distribution mapping \eqref{3.5} to the space $fsvr\,{\mathcal M}_{d}(\Hs)$, respectively to $fosvr\,{\mathcal M}_{d}(\Hs)$ and more particular to $\Ec^{0}_{d}(\Hs)$. Since
$$\tau_{\xi}(A,B)=(\xi\otimes\xi)(A,B)=[\xi(A),\xi(B)]_{\Hs},\quad A, B\in\widetilde{\Bc}or(\mathbb{R}^{d})$$
we can apply Lemma 19(1) and (3) of \cite[pp. 66]{Kaki}, from where we deduce that $\tau_{\xi}\in fsvr\,\mathfrak{M}_{2d}^{pd}\big(\Cs_{1}(H)\big) $, respectively $\tau_{\xi}\in fosvr\,\mathfrak{M}_{2d}^{pd}\big(\Cs_{1}(H)\big)$.

It is not difficult to see that it results

\begin{prop}\label{prop2.1}
Statements analogous as for the operator covariance distribution mapping \eqref{3.5} in the corollaries above, hold for the corresponding mappings
\begin{equation}\label{311}
\begin{split}
fosvr\,{\mathcal M}_{d}(\Hs)\ni \xi&\mapsto\tau_{\xi}\in fosvr\,\mathfrak{M}_{2d}^{pd}\big(\Cs_{1}(H)\big)\\
fsvr\,{\mathcal M}_{d}(\Hs)\ni \xi&\mapsto\tau_{\xi}\in fsvr\,\mathfrak{M}_{2d}^{pd}\big(\Cs_{1}(H)\big).
\end{split}
\end{equation}

Moreover, these operator covariance bimeasure mappings are natural extensions of the covariance function mapping
\begin{equation}\label{313}
\Ec^{0}_{d}(\Hs)\ni F \mapsto \Gamma_{F}\in pd \Ec^{0}_{2d}\big(\Cs_{1}(H)\big),
\end{equation}
associated to classical m.s.o. random fields, as was defined in \cite[Section 4.1, pp. 148]{Kaki}.
\end{prop}

Now it is naturally to ask how the mappings (\ref{313}),(\ref{311}) can be regarded as restrictions of the covariance distribution mapping \eqref{3.5} to the first three subspaces from \eqref{3.4}, i.e our generalization is coherent to the classical case in \cite{Kaki}.\\
We shall show that the operator covariance distributions of \eqref{311} and of \eqref{313} will be regarded as $\Cs_1(H)$-valued distribution $C^{\tau_\xi}$ on $\R^{2d}$ corresponding to (generated by) the bimeasure $\tau_\xi$, respectively $C^{\Gamma_F}$ corresponding to (generated by) the correlation function $\Gamma_F$. More precisely it holds
\begin{prop}
The positive definite operator valued bimeasures from \eqref{311} and the operator covariance function of \eqref{313} satisfy
\begin{equation}\label{314}
C_{U^{\xi}}=C^{\tau_{\xi}} \ \text{ and } \
C_{U^{F}}=C^{\Gamma_{F}},
\end{equation}
respectively.
\end{prop}
\begin{proof}
Let  $\xi\in fosvr\,{\mathcal M}_{d}(\Hs)$, and $U^{\xi}\in(\Dc_{d}^{m})'(\Hs)$ given by \eqref{3.3}. Then we successively obtain for each $\varphi,\psi \in \Dc_d^m(\Hs)$

\begin{eqnarray*}
C_{U^{\xi}}(\varphi\otimes\overline{\psi})&=&\big[U^{\xi}\varphi,U^{\xi}\psi\big]_{\Hs}=\left[\int\limits_{\mathbb{R}^{d}}\varphi(s)d\xi(s),
\int\limits_{\mathbb{R}^{d}}\psi(t)d\xi(t)\right]_{\Hs}=\nonumber\\
&=&\int\limits_{\mathbb{R}^{d}}\int\limits_{\mathbb{R}^{d}}\varphi(s)\overline{\psi(t)} d\big(\xi\otimes\xi\big)(s,t)=\nonumber\\ &=&\int\limits_{\mathbb{R}^{d}}\int\limits_{\mathbb{R}^{d}}\varphi(s)\overline{\psi(t)} d\tau_{\xi}(s,t),
\end{eqnarray*}
which by \eqref{eq:covdistrib} gives the first relation in \eqref{314}.

In particular for $\xi=\xi^{F}$ with $\xi^{F}$ given by \eqref{3.2}, we have first $U^{\xi}=U^{F}$, given by \eqref{3.3} and secondly, for each $\varphi,\,\psi\in \Dc_{d}^{m}$

\begin{eqnarray*}
C^{\tau_{\xi^{F}}}(\varphi\otimes\overline{\psi})&=&\int\limits_{\mathbb{R}^{d}}\int\limits_{\mathbb{R}^{d}}\varphi(s)\overline{\psi(t)}\Gamma_{F}(s,t) ds dt \nonumber\\
&=& C^{\Gamma_{F}} (\varphi\otimes\overline{\psi}),
\end{eqnarray*}
which finally means the second relation from \eqref{314}.
\end{proof}

Finally let's observe that for the particular classes of stochastic mappings considered in this Section a Wold type decomposition holds, which is closely connected to the Wold decomposition given at the end of the previous Section.

For $F\in \Ec_d^0 (\Hs)$ the \emph{observable structure} $\Hs_F^{t_0}$ (respectively $\Hs_{(F)}^{t_0}$) is the closed $\Bc(H)$-module (space) in $\Hs$ generated by the set $\{ F(t),\ t\le t_o \}$, while for $\xi \in fosvr {\mathcal M}_d(\Hs)$, $\Hs_\xi^{t_0}$, ($\Hs_{(\xi)}^{t_0}$) is the closed $\Bc(H)$-module (space) in $\Hs$ generated by the set $\{ \xi(A),\ A\in\widetilde\Bs or (\R^d)\ :\ t \in A \Rightarrow t\le t_0 \}$.

Using the observable structures just defined one can easily infer corresponding Wold decompositions for $F$ and $\xi$ as in Theorem \ref{thm:WoldRDF}:
\begin{equation}\label{eq:Wold.F}
F = F^{det} + F^p
\end{equation}
respectively
\begin{equation} \label{eq:Wold.xi}
\xi = \xi^{det} + \xi^p .
\end{equation}

Since, as it is not hard to see, the above definitions of the observable structures are coherent to the classical ones, i.e.
\begin{equation*}
\Hs_F^{t_0} = \Hs^{t_0}_{U^F} (\Hs_{(F)}^{t_0} = \Hs_{(U^F)}^{t_0})\  \text{ and }\
\Hs_\xi^{t_0} = \Hs_{U^\xi}^{t_0} (\Hs_{(\xi)}^{t_0} = \Hs_{(U^\xi)}^{t_0}),
\end{equation*}
we obtain the following coherence result for the three types of Wold decomposition (Theorem \ref{thm:WoldRDF}, \eqref{eq:Wold.F} and \eqref{eq:Wold.xi}).

\begin{cor}
The m.s.o.r.d.f.'s $F, \ \xi^F,\ U^F$ as well as $\xi,\ U^\xi$ are simultaneously (operator) deterministic, (operator) nondeterministic or (operator) purely nondeterministic. Also the corresponding summands in the Wold decomposition are the same, i.e.
\[
(U^F)^{det} = U^{F^{det}},\quad (U^F)^p = U^{F^p},
\]
as well as
\[
(U^\xi)^{det} = U^{\xi^{det}} \quad (U^\xi)^p = U^{\xi^p}.
\]
\end{cor}

\section{Concluding remarks}

Let's mention that the invariance to the gramian unitary equivalences is an important property not only for a pair of m.s.o. random distribution fields with the same operator covariance distribution (Corollary \ref{cor3.3}), but for a larger class of m.s.o. random distribution fields, namely those having the property of operator stationarity.
Such a study of stationarity, stationarily cross correlatedness, the corresponding spectral representations, as well as the periodically correlatedness are treated in \cite{GaPoMSOSMII} and \cite{GaSidJFA}.
An investigation of some other classes of m.s.o. random distribution fields, which are invariant to similarities (e.g. uniformly bounded linearly stationary m.s.o. random distribution fields or other frameworks as in \cite{ValPGa}), or which are invariant to the actions of bounded $\Bc(H)$-linear operators (especially harmonizable m.s.o. random distribution fields) and where the concept of propagator, see \cite{MasDilat} or \cite{Szaf}, as well as some ``intertwining'' properties of the Fourier transform with the covariance distribution mapping \eqref{3.5}, or with the covariance distribution bimeasure mappings \eqref{311}, \eqref{313} play important roles, will be conducted in some forthcoming papers.

\end{document}